\documentclass[a4paper,12pt,reqno]{article}
\usepackage[latin1]{inputenc}
\usepackage[T1]{fontenc}
\usepackage{lmodern}
\usepackage[english]{babel}
\usepackage{amsmath}
\usepackage{amssymb}
\usepackage{amsthm}
\usepackage{bm}
\usepackage[pagestyles]{titlesec}\usepackage[pdfpagemode=UseNone,pdfstartview=FitH]{hyperref}
\renewpagestyle{plain}{\setfoot[\thepage][][]{}{}{\thepage}}
\theoremstyle{plain}
\newtheorem{theorem}{Theorem}[section]
\newtheorem{corollary}{Corollary}[section]
\newtheorem{lemma}{Lemma}[section]

\theoremstyle{definition}
\newtheorem{definition}{Definition}[section]
\newtheorem{example}{Example}[section]
\newtheorem{remark}{Remark}[section]
\numberwithin{equation}{section}
\allowdisplaybreaks[1]
  
\newcommand*{\Zset}{\mathbb{Z}}  
  
\newcommand*{\Rset}{\mathbb{R}}  
\newcommand*{\Cset}{\mathbb{C}}
\newcommand*{\mx}[1]{\bm{#1}} 
\begin{document}
\title{\textbf{On the positive definiteness and eigenvalues of meet and join matrices}}
\author{}
\date{19.9.2012}
\maketitle
\begin{center}
\textsc{Mika Mattila$^*$ and Pentti Haukkanen}\\
School of Information Sciences\\
FI-33014 University of Tampere, Finland\\[5mm]
\end{center}
\begin{abstract}
In this paper we study the positive definiteness of meet and join matrices using a novel approach. When the set $S_n$ is meet closed, we give a sufficient and necessary condition for the positive definiteness of the matrix $(S_n)_f$. From this condition we obtain some sufficient conditions for positive definiteness as corollaries. We also use graph theory and show that by making some graph theoretic assumptions on the set $S_n$ we are able to reduce the assumptions on the function $f$ while still preserving the positive definiteness of the matrix $(S_n)_f$. Dual theorems of these results for join matrices are also presented. As examples we consider the so-called power GCD and power LCM matrices as well as MIN and MAX matrices. Finally we give bounds for the eigenvalues of meet and join matrices in cases when the function $f$ possesses certain monotonic behaviour.
\end{abstract}
\emph{$^\ast$Corresponding\ author.}\ \textit{Tel.:}\ +358\ 50\ 318\ 5881,\ \textit{fax:}\ +358\ 3\ 219\ 1001\\
\emph{E-mail addresses:}\ mika.mattila@uta.fi\ (M. Mattila),\\\hspace*{32mm} pentti.haukkanen@uta.fi\ (P. Haukkanen)

\newpage

\section{Introduction}

The research of GCD and LCM matrices was initiated by H. J. S. Smith \cite{S} in 1875 when he studied the determinant of the $n\times n$ matrix in which the $ij$ element is the greatest common divisor $(i,j)$ of $i$ and $j$. He also considered the $n\times n$ matrix with the least common multiple $[i,j]$ of $i$ and $j$ as its $ij$ element. During the next century the determinants of GCD type matrices were a topic of interest for many number theorists and linear algebraists (see the references in \cite{HWS}; the two articles \cite{L} and \cite{Wi} by Lindström and Wilf are especially relevant). In 1989 Beslin and Ligh \cite{BL} initiated a new wave of more intense research of GCD matrices, which soon led to poset-theoretic generalizations of GCD matrices. Rajarama Bhat \cite{RB} gave the definition of meet matrix, and Haukkanen \cite{H} was the first to study these matrices systematically. Join matrices were defined later by Korkee and Haukkanen \cite{KH}. 

Over the years many authors have considered the positive definiteness of GCD, LCM, meet and join matrices. In 1989 Beslin and Ligh \cite{BeL} showed that the GCD matrix $(S)$ of the set $S=\{x_1,\ldots,x_n\}$, in which the $ij$ element is $(x_i,x_j)$, is positive definite. Four years later Bourque and Ligh \cite{BL2} proved that if $f$ is an arithmetical function such that
\[
d\,|\,x_i\text{ for some }x_i\in S\Rightarrow(f*\mu)(d)>0,
\]
then the GCD matrix $(S)_f$ with $f((x_i,x_j))$ as its $ij$ element is positive definite. In \cite{BL} Beslin and Ligh report results concerning the positive definiteness of GCD matrices with respect to generalized Ramanujan's sums. In 2001 Korkee and Haukkanen \cite{KH2} gave a sufficient condition for positive definiteness of meet matrices, and in \cite{KH} they presented a similar condition for join matrices. A couple of years later Altinisik et al. \cite{AST} obtained a sufficient and necessary condition for positive definiteness of a matrix closely related to meet matrices. At the same time Ovall \cite{O} went back to GCD and LCM matrices and showed that GCD and certain reciprocal matrices are positive definite, whereas some reciprocal matrices and certain LCM matrices are indefinite. In 2006 Bhatia \cite{Bh2} showed once again that the usual GCD matrix is infinitely divisible and therefore positive definite. Later Bhatia \cite{Bh} also studied certain MIN matrices and presented six proofs for their positive definiteness (it should be noted that MIN matrices can easily be seen as special cases of meet matrices). 

There are also some results for the eigenvalues of GCD-type matrices to be found in the literature. Wintner \cite{W} published results concerning the largest eigenvalue of the $n\times n$ matrix having
\[
\left(\frac{(i,j)}{[i,j]}\right)^\alpha
\]
as its $ij$ entry and subsequently Lindqvist and Seip \cite{LS} investigated the asymptotic behaviour of the smallest and largest eigenvalue of the same matrix. More recently Hilberdink \cite{Hi} and also Berkes and Weber \cite{BW} addressed this same topic from an analytical perspective. The first paper concerning the eigenvalues of proper GCD matrices was by Balatoni \cite{B} as he considered the eigenvalues of the classical Smith's GCD matrix.

One way to obtain information about the eigenvalues of GCD type matrices is to study the norms of these matrices. The $O$ estimates of the norms have been studied in many papers, see \cite{A, ATH2, Ha1, Ha3, Ha2}. Hong and Loewy \cite{HL, HL2} studied the asymptotic behaviour of a special kind of GCD matrices, Altinisik \cite{A2} provides information about the eigenvalues of GCD matrices, a paper by Hong and Enoch Lee \cite{HE} addresses the eigenvalues of reciprocal LCM matrices and there is also one paper about the eigenvalues of meet and join matrices by Ilmonen et al. \cite{IHM}. 

In this paper we provide new information about the positive definiteness and the eigenvalues of meet and join matrices. The notations and most of the concepts are defined in Section 2. Section 3 contains some new characterizations and key examples of positive definite meet and join matrices. In Section 4 we make use of graph theory and study the positive definiteness of meet and join matrices from this new graph theoretic perspective. In Section 5 we provide upper bounds for all the eigenvalues of meet and join matrices in which the function $f$ evinces certain monotonic behaviour. 

\section{Preliminaries}

Throughout this paper $(P,\preceq)$ is an infinite but locally finite lattice, $f:P\to\Rset$
is a real-valued function on $P$ and $(x_n)_{n=1}^\infty$ is an infinite sequence of distinct elements of $P$ such that
\begin{equation}\label{eq:condition}
x_i\preceq x_j\Rightarrow i\leq j.
\end{equation}

For every $n\in\Zset_+$, let $S_n=\{x_1,x_2,\ldots,x_n\}$. The set $S_n$ is said to be \emph{meet closed} if $x_i\wedge x_j\in S_n$ for all $x_i,x_j\in S_n$, in other words, the structure $(S_n,\preceq)$ is a meet semilattice. The concept of \emph{join closed set} is defined dually.

The $n\times n$ matrix having $f(x_i\wedge x_j)$ as its $ij$ element is the \emph{meet matrix} of the set $S_n$ with respect to $f$ and is denoted by $(S_n)_f$. Similarly, the $n\times n$ matrix having $f(x_i\vee x_j)$ as its $ij$ element is the \emph{join matrix} of the set $S_n$ with respect to $f$ and is denoted by $[S_n]_f$. When $(P,\preceq)=(\Zset_+,|)$ the matrices $(S_n)_f$ and $[S_n]_f$ are referred to as the GCD and LCM matrices of the set $S_n$ with respect to $f$.

Let $D_n=\{d_1,d_2,\ldots,d_{m_n}\}$ be any finite subset of $P$ containing all the elements $x_i\wedge x_j$, where $x_i,x_j\in S_n$, and having its elements arranged so that
\[
d_i\preceq d_j\Rightarrow i\leq j.
\]
Next we define the function $\Psi_{D_n,f}$ on $D_n$ inductively as
\begin{equation}
\Psi_{D_n,f}(d_k)=f(d_k)-\sum_{d_v\prec d_k}\Psi_{D_n,f}(d_v),
\label{eq:Psi1}
\end{equation}
or equivalently
\begin{equation}
f(d_k)=\sum_{d_v\preceq d_k}\Psi_{D_n,f}(d_v).
\label{eq:Psi2}
\end{equation}
Thus we have
\begin{equation}
\Psi_{D_n,f}(d_k)=\sum_{d_v\preceq d_k}f(d_v)\mu_{D_n}(d_v,d_k),
\label{eq:Psi3}
\end{equation}
where $\mu_{D_n}$ is the Möbius function of the poset $(D_n,\preceq)$, see \cite[Section IV.1]{Aig} and \cite[Proposition 3.7.1]{St}.

Let $E_{D_n}$ denote the $n\times m_n$ matrix defined as
\begin{equation}\label{eq:E}
(e_{D_n})_{ij}=\left\{
 \begin{array}{cc}
    1 & \textrm{if }d_{j}\preceq x_{i}\textrm{,} \\
    0 & \textrm{otherwise.}
 \end{array}
\right.
\end{equation}
The matrix $E_{D_n}$ may be referred to as the incidence matrix of the set $D_n$ with respect to the set $S_n$ and the partial ordering $\preceq$.

The set $D_n$, the function $\Psi_{D_n,f}$ and the matrix $E_{D_n}$ are needed when considering the matrix $(S_n)_f$. Next we define the dual concepts which we use in the study of the matrix $[S_n]_f$.

Let $B_n=\{b_1,b_2,\ldots,b_{l_n}\}$ be any finite subset of $P$ containing all the elements $x_i\vee x_j$ with $x_i,x_j\in S_n$ and having its elements arranged so that
\[
b_i\preceq b_j\Rightarrow i\leq j.
\]
We define the function $\Phi_{B_n,f}$ on $B_n$ inductively as
\begin{equation}
\Phi_{B_n,f}(b_k)=f(b_k)-\sum_{b_k\prec b_v}\Phi_{B_n,f}(b_v),
\label{eq:Phi1}
\end{equation}
or equivalently
\begin{equation}
f(b_k)=\sum_{b_k\preceq b_v}\Phi_{B_n,f}(b_v).
\label{eq:Phi2}
\end{equation}
Thus we have
\begin{equation}
\Phi_{B_n,f}(b_k)=\sum_{b_k\preceq b_v}f(b_v)\mu_{B_n}(b_k,b_v),
\label{eq:Phi3}
\end{equation}
where $\mu_{B_n}$ is the Möbius function of the poset $(B_n,\preceq)$.

Let $E_{B_n}$ denote the $n\times l_n$ matrix defined as
\begin{equation}\label{eq:E2}
(e_{B_n})_{ij}=\left\{
 \begin{array}{cc}
    1 & \textrm{if }b_{j}\succeq x_{i}\textrm{,} \\
    0 & \textrm{otherwise.}
 \end{array}
\right.
\end{equation}
We refer to the matrix $E_{B_n}$ as the incidence matrix of the set $B_n$ with respect to the set $S_n$ and the partial ordering $\succeq$.

\begin{remark}\label{re:order}
If we are only interested in the positive definiteness and eigenvalues of meet and join matrices, then the condition \eqref{eq:condition} is, in fact, not necessary but can still be made without restricting generality. If $S_n$ does not satisfy the condition \eqref{eq:condition} and $S_n'$ is a set obtained from $S_n$ by rearranging its elements so that \eqref{eq:condition} holds, then there exists a permutation matrix $P$ such that
\[
(S_n')_f=P(S_n)_fP^T=P(S_n)_fP^{-1}.
\]
Thus the matrices $(S_n')_f$ and $(S_n)_f$ are similar and therefore have the same eigenvalues, positive definiteness properties etc. 
\end{remark}

It is well known (see, for example \cite{ATH, MH}) that adopting the above notations the matrices $(S_n)_f$ and $[S_n]_f$ can be factored as
\begin{equation}\label{eq:fac}
(S_n)_f=E_{D_n}\Lambda_{D_n,f}E_{D_n}^T\quad\text{and}\quad [S_n]_f=E_{B_n}\Delta_{B_n,f}E_{B_n}^T,
\end{equation}
where $$\Lambda_{D_n,f}=\text{diag}(\Psi_{D_n,f}(d_1),\Psi_{D_n,f}(d_2),\ldots,\Psi_{D_n,f}(d_{m_n}))$$ and $$\Delta_{B_n,f}=\text{diag}(\Phi_{B_n,f}(b_1),\Phi_{B_n,f}(b_2),\ldots,\Phi_{B_n,f}(b_{l_n})).$$ By using the first factorization in a case when the set $S_n$ is meet closed, it is easy to show (see \cite[Theorem 4.2]{ATH}) that 
\begin{equation}\label{eq:det1}
\det (S_n)_f=\Psi_{S_n,f}(x_1)\Psi_{S_n,f}(x_2)\cdots\Psi_{S_n,f}(x_n).
\end{equation}
Similarly, when the set $S_n$ is join closed we have
\begin{equation}\label{eq:det2}
\det [S_n]_f=\Phi_{S_n,f}(x_1)\Phi_{S_n,f}(x_2)\cdots\Phi_{S_n,f}(x_n)
\end{equation}
(see \cite[Theorem 4.2]{MH}). In the next section these determinant formulas appear also to be useful when considering the positive definiteness of meet and join matrices. 

\section{On the positive definiteness of meet and join matrices}

We begin our study by considering the positive definiteness of the matrix $(S_n)_f$ in case when the set $S_n$ is meet closed. Under these circumstances we are able to give sufficient and necessary conditions for positive definiteness of the matrix $(S_n)_f$. Theorem \ref{th:posdefmeetcl} is also closely related to Theorem 5.1 in \cite{AST}.

\begin{theorem}\label{th:posdefmeetcl}
If the set $S_n$ is meet closed, then the matrix $(S_n)_f$ is positive definite if and only if $\Psi_{S_n,f}(x_i)>0$ for all $i=1,2,\ldots,n$.
\end{theorem}
\begin{proof}
Since removing a maximal element does not affect the meet closeness of the set $S_i$, it follows that all the sets $S_n,S_{n-1},\ldots,S_2,S_1$ are meet closed. In addition, the determinants of the matrices $(S_i)_f$, where $i=1,2,\ldots,n$, are the leading principal minors of the matrix $(S_n)_f$. By \eqref{eq:det1} we have
\begin{align*}
\det(S_1)_f&=\Psi_{S_n,f}(x_1)\\
\det(S_2)_f&=\Psi_{S_n,f}(x_1)\Psi_{S_n,f}(x_2)\\
&\vdots\\
\det(S_{n-1}&)_f=\Psi_{S_n,f}(x_1)\Psi_{S_n,f}(x_2)\cdots\Psi_{S_n,f}(x_{n-1})\\
\det(S_n)_f&=\Psi_{S_n,f}(x_1)\Psi_{S_n,f}(x_1)\cdots\Psi_{S_n,f}(x_{n-1})\Psi_{S_n,f}(x_n).
\end{align*}
Now $(S_n)_f$ is positive definite if and only if $\det(S_i)_f>0$ for all $i=1,2,\ldots,n$ (see \cite[Theorem 7.2.5]{HJ}), and the determinants above are all positive if and only if $\Psi_{S_n,f}(x_i)>0$ for all $i=1,2,\ldots,n$.
\end{proof}

Next we present a dual theorem for join matrices.

\begin{theorem}\label{th:posdefjoincl}
If the set $S_n$ is join closed, then the matrix $[S_n]_f$ is positive definite if and only if $\Phi_{S_n,f}(x_i)>0$ for all $i=1,2,\ldots,n$.
\end{theorem}

\begin{proof}
Let us denote \[S_1'=\{x_n\},\ S_2'=\{x_{n-1},x_n\},\ldots,S_{n-1}'=\{x_2,\ldots,x_{n-1},x_n\}.\] Since the determinants of the matrices \[[S_1']_f,[S_2']_f,\ldots,[S_{n-1}']_f\ \text{and}\ [S_n]_f\]
constitute a nested sequence of $n$ principal minors of $[S_n]_f$, the matrix $[S_n]_f$ is positive definite if and only if all of these matrices have positive determinants (again, see \cite[Theorem 7.2.5]{HJ}). Since all these sets are join closed, the determinants can be calculated by using \eqref{eq:det2}. The rest of the proof is similar to the proof of Theorem \ref{th:posdefmeetcl}.
\end{proof}

\begin{example}\label{ex:chain}
Let $S_n$ be a chain. Thus $x_1\prec x_2\prec\cdots\prec x_{n-1}\prec x_n$. Clearly, the set $S_n$ is both meet and join closed (the matrices $(S_n)_f$ and $[S_n]_f$ may be referred to as the MIN and MAX matrices of the chain $S_n$ respectively). In this case we have $\Psi_{S_n,f}(x_1)=f(x_1)$ and
\[\Psi_{S_n,f}(x_i)=\sum_{x_k\preceq x_i}f(x_k)\mu_{S_n}(x_k,x_i)=f(x_i)-f(x_{i-1})\]
for all $i=2,\ldots,n$. Now it follows from Theorem \ref{th:posdefmeetcl} that the matrix $(S_n)_f$ is positive definite if and only if $f(x_1)>0$ and $f(x_i)>f(x_{i-1})$ for all $i=2,\ldots,n$. In other words, we must have
\[0<f(x_1)<f(x_2)<\cdots<f(x_{n-1})<f(x_n).\]

If we set $(P,\preceq)=(\Zset^+,\leq)$, $f(k)=k$ for all $k\in\Zset$ and $S_n=\{1,2,\ldots,n\}$, we obtain the MIN matrix studied recently by Bhatia \cite{Bh}. Among other things, he presents six distinct proofs for the positive definiteness of this matrix. The one in this example is yet another different proof.

Similarly, by using Theorem \ref{th:posdefjoincl} it is possible to show that the matrix $[S_n]_f$ is positive definite if and only if \[0<f(x_n)<f(x_{n-1})<\cdots<f(x_2)<f(x_1).\]
\end{example}

Next we focus on the case when the set $S_n$ is neither meet nor join closed. It turns out that by using Theorems \ref{th:posdefmeetcl} and \ref{th:posdefjoincl} it is possible to say something about the positive definiteness of the matrices $(S_n)_f$ and $[S_n]_f$ also under these circumstances. Corollary \ref{th:posdefmeet} may be seen as a generalization of Theorem 1 (i) in \cite{BL}.

\begin{corollary}\label{th:posdefmeet}
Let $D_n$ be any finite meet closed subset of $P$ containing all the elements of $S_n$. If $\Psi_{D_n,f}(d_i)>0$ for all $d_i\in D_n$, then the matrix $(S_n)_f$ is positive definite.
\end{corollary}
\begin{proof}
By Theorem \ref{th:posdefmeetcl} the matrix $(D_n)_f$ is positive definite. Thus the matrix $(S_n)_f$ is also positive definite since it is a principal submatrix of the matrix $(D_n)_f$.
\end{proof}

\begin{example}\label{ex:powergcd}
Let $(P,\preceq)=(\Zset^+,|)$, $\alpha\in\Rset$ and $f(n)=n^\alpha$ for all $n\in\Zset^+$. Under these assumptions the matrices $(S_n)_f$ and $[S_n]_f$ become the so-called power GCD and LCM matrices, which have been studied extensively by Hong et al. \cite{HE, HL}. It is well known that the matrix $(S_n)_f$ is positive definite if $\alpha>0$ (see \cite[Example 1]{BL2} and \cite[Example 3]{BL}). Here we give another proof for this by using the previous corollary.

Let $$D_n=\downarrow\hspace{-1mm}S_n=\{d\in\Zset^+\ \big|\ d\,|\,x_i\ \text{for some}\ x_i\in S_n\}.$$
Let $\ast$ denote the Dirichlet convolution and $\mu$ denote the number-theoretic Möbius function. Now for every $d_k\in D_n$ we have
\[
\Psi_{D_n,f}(d_k)=\sum_{d_v\,|\,d_k}d_v^\alpha\mu\left(\frac{d_k}{d_v}\right)=(f\ast\mu)(d_k)=J_\alpha(d_k)
=d_k^\alpha\prod_{p\,|\,d_k}\left(1-\frac{1}{p^{\alpha}}\right),
\]
where $J_\alpha$ is a generalization of the Jordan totient function. If $\alpha>0$, then clearly $J_\alpha(d_k)>0$ for all $d_k\in D_n$ and therefore by Corollary \ref{th:posdefmeet} the matrix $(S_n)_f$ is positive definite.
\end{example}

Next we present a similar corollary that concerns the matrix $[S_n]_f$. The proof is essentially the same as the proof of Corollary \ref{th:posdefmeet}.

\begin{corollary}\label{th:posdefjoin}
Let $B_n$ be any finite join closed subset of $P$ containing all the elements of $S_n$. If $\Phi_{B_n,f}(b_i)>0$ for all $b_i\in B_n$, then the matrix $[S_n]_f$ is positive definite.
\end{corollary}

\begin{example}\label{powerlcm}
Let $(P,\preceq)=(\Zset^+,|)$ as in Example \ref{ex:powergcd}, $\alpha\in\Rset^+$ and $f(n)=\frac{1}{n^\alpha}$ for all $n\in\Zset^+$.
Hong and Enoch Lee \cite[Theorem 2.1]{HE} have shown that the matrix $[S_n]_f$ is positive definite. Here we present a different
proof for this fact by using Corollary \ref{th:posdefjoin}.

Let $\alpha>0$, let $\downarrow\hspace{-1mm}\text{lcm}\,S_n$ denote the set of divisors of $\text{lcm}\,S_n$ and let $\uparrow\hspace{-1mm}S_n$ stand for the set $\{k\in\Zset^+\,\big|\,x_i|k\text{ for some }i=1,\ldots,n\}.$ Now let $$B_n=\uparrow\hspace{-1mm}S_n\cap\downarrow\hspace{-1mm}\text{lcm}\,S_n=\{d\in\Zset^+\ \big|\ x_i\,|\,d\ \text{for some}\ x_i\in S_n\ \text{and}\ d\,|\,\text{lcm}\,S_n\}.$$ Then for every $b_k\in B_n$ we have
\begin{align*}
\Phi_{B_n,f}(b_k)&=\sum_{b_k\,|\,b_v\,|\,\text{lcm}S_n}\frac{1}{b_v^\alpha}\,\mu\left(\frac{b_v}{b_k}\right)\\
&=\left(\frac{1}{\text{lcm}S_n}\right)^\alpha\sum_{b_k\,|\,b_v\,|\,\text{lcm}S_n}\left(\frac{\text{lcm}S_n}{b_v}\right)^\alpha\mu\left(\frac{b_v}{b_k}\right)\\
&=\left(\frac{1}{\text{lcm}S_n}\right)^\alpha\sum_{a\,|\,\frac{\text{lcm}S_n}{b_k}}\left(\frac{(\text{lcm}S_n)/b_k}{a}\right)^\alpha\mu\left(a\right)\\
&=\left(\frac{1}{\text{lcm}S_n}\right)^\alpha J_\alpha\left(\frac{\text{lcm}S_n}{b_k}\right)>0.
\end{align*}
Thus by Corollary \ref{th:posdefjoin} the matrix $[S_n]_f$ is positive definite.
\end{example}

As seen in the above examples, there are two obvious ways to choose the sets $D_n$ and $B_n$. The first is to take $D_n$ (resp. $B_n$) to be the meet (resp. join) subsemilattice of $P$ generated by the set $S_n$. The other is to take $D_n=\downarrow\hspace{-1mm}S_n$ and $B_n=\uparrow\hspace{-1mm}S_n$ in case when the sets $\downarrow\hspace{-1mm}S_n$ and $\uparrow\hspace{-1mm}S_n$ are finite, and otherwise take $D_n=\downarrow\hspace{-1mm}S_n\cap\uparrow\hspace{-1mm}(\wedge S_n)$ and $B_n=\uparrow\hspace{-1mm}S_n\cap\downarrow\hspace{-1mm}(\vee S_n)$ (here $\vee S_n=x_1\vee\cdots\vee x_n$ and $\wedge S_n=x_1\wedge\cdots\wedge x_n$). Benefits of both choices are explained in \cite{MH}.

Although Corollaries \ref{th:posdefmeet} and \ref{th:posdefjoin} can be used in many cases, their conditions are not necessary for the positive definiteness of the matrices $(S_n)_f$ and $[S_n]_f$ and thus they are not always applicable. The following example illustrates this.

\begin{example}
Let $(P,\preceq)=(\Zset^+,|)$, $S_3=\{6,10,15\}$ and
\[\left\{
\begin{array}{cc}
  f(1)=&0\\
  f(2)=&-1\\
  f(3)=&3\\
  f(5)=&-2\\
  f(6)=&5\\
  f(10)=&2\\
  f(15)=&3.
 \end{array}
\right.
\]
Then we obtain the GCD matrix
\[
(S_3)_f=\begin{bmatrix}
5 & -1 & 3\\
-1 & 2 & -2\\
3 & -2 & 3
\end{bmatrix},
\]
which can easily be shown to be positive definite. However, if we choose the elements of $D_3$ as in Figure \ref{fig: kuva3}, direct calculations show that $\Psi_{D_3,f}(d_2)=-1<0$ and $\Psi_{D_3,f}(d_4)=-2<0$. Thus the meet matrix $(S_3)_f$ is positive definite although some of the values of $\Psi_{D_3,f}$ are negative.
\begin{figure}[ht]
\centering
\setlength{\unitlength}{0.8 cm}
\begin{picture}(8,6)
\thicklines
\put(4,1){\line(0,1){2}}
\put(4,1){\line(1,1){2}}
\put(4,1){\line(-1,1){2}}
\put(2,3){\line(1,1){2}}
\put(2,3){\line(0,1){2}}
\put(4,3){\line(1,1){2}}
\put(4,3){\line(-1,1){2}}
\put(6,3){\line(0,1){2}}
\put(6,3){\line(-1,1){2}}

\put(4,0.5){$d_1=1$}
\put(3,3.3){$d_3=3$}
\put(0.7,2.5){$d_2=2$}
\put(6.2,2.75){$d_4=5$}
\put(1,5.3){$d_5=6$}
\put(3,5.3){$d_6=10$}
\put(5,5.3){$d_7=15$}

\put(4,1){\circle*{0.2}}
\put(4,3){\circle*{0.2}}
\put(2,3){\circle*{0.2}}
\put(6,3){\circle*{0.2}}
\put(2,5){\circle*{0.2}}
\put(4,5){\circle*{0.2}}
\put(6,5){\circle*{0.2}}
\end{picture}
\caption{The lattice $(D_3,\preceq)$ and the choices of the elements of $D_3$. }
\label{fig: kuva3}
\end{figure}
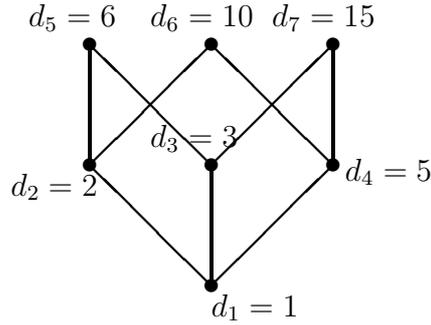
\end{example}

\section{Trees, $A$-sets and positive definiteness}

Next we turn our attention to the special case where the Hasse diagram of the set $\text{meetcl}(S_n)$ is a tree (when it is considered as an undirected graph). Here $\text{meetcl}(S_n)$ (resp. $\text{joincl}(S_n)$) denotes the meet subsemilattice (resp. the join subsemilattice) of $P$ generated by the set $S_n$. Like in Example \ref{ex:chain}, also in this case a certain monotonicity property of $f$ guarantees the positive definiteness of $(S_n)_f$ (resp. $[S_n]_f$). First we present the definitions of these properties.

\begin{definition}
The set $S_n\subseteq P$ is said to be a \emph{$\wedge$-tree set} if the Hasse diagram of $\text{meetcl}(S_n)$ is a tree. Analogously, $S_n$ is a \emph{$\vee$-tree set} if the Hasse diagram of $\text{joincl}(S)$ is a tree.
\end{definition}

There are also a couple of other characterizations for $\wedge$-tree sets and $\vee$-tree sets. We present these only for $\wedge$-tree sets, since the characterizations for $\vee$-tree sets are dual to these.

\begin{lemma}\label{lemma1}
The following statements are equivalent:
\begin{enumerate}
\item $S_n$ is a $\wedge$-tree set.
\item Every element in $\mathrm{meetcl}(S_n)$ covers at most one element of $\mathrm{meetcl}(S_n)$.
\item For every $x\in\mathrm{meetcl}(S_n)$ the set
\[
(\downarrow\hspace{-1mm}x)\cap\mathrm{meetcl}(S_n)=\{y\in\mathrm{meetcl}(S_n)\ \big|\ y\preceq x\}
\]
is a chain.
\item For all $x,y,z\in\mathrm{meetcl}(S_n)$ we have
\[
(x\preceq z\ \mathrm{and}\ y\preceq z)\Rightarrow (x\preceq y\ \mathrm{or}\ y\preceq x).
\]
\end{enumerate}
\end{lemma}
\begin{proof}
The proof is simple and straightforward.
\end{proof}

Next we define the monotonicity property for $f$ that we mentioned earlier.

\begin{definition}\label{def:order-preserving}
The function $f:P\to\Rset$ is \emph{strictly order-preserving} if
\begin{equation}\label{eq:order-preserving}
x\prec y\Rightarrow f(x)<f(y).
\end{equation}
Analogously, $f$ is \emph{strictly order-reversing} if
\begin{equation}\label{eq:order-reversing}
x\prec y\Rightarrow f(y)<f(x).
\end{equation}
The function $f$ is said to be \emph{order-preserving} (resp. \emph{order-reversing}) if equality is allowed on the right side of \eqref{eq:order-preserving} (resp. \eqref{eq:order-reversing}).
\end{definition}

\begin{remark}\label{re:chain}
After adopting the terminology in Definition \ref{def:order-preserving} we are able to revisit Example \ref{ex:chain} and express its results in the following form: If the set $S_n$ is a chain, then
\begin{enumerate}
\item $(S_n)_f$ is positive definite\\ $\Leftrightarrow$ $f$ is strictly order-preserving in $S_n$ with positive values,
\item $[S_n]_f$ is positive definite\\ $\Leftrightarrow$ $f$ is strictly order-reversing in $S_n$ with positive values.
\end{enumerate}
\end{remark}

The following theorem presents a condition for positive definiteness of $(S_n)_f$ (resp. $[S_n]_f$) in the case when the values of $f$ are positive and $f$ is order-preserving (resp. order-reversing).

\begin{theorem}\label{th:posdeftree}
Let $f(x)>0$ for all $x\in P$. Then the following statements hold:
\begin{enumerate}
\item If $S_n$ is a $\wedge$-tree set and $f$ is strictly order-preserving, then $(S_n)_f$ is positive definite.
\item If $S_n$ is a $\vee$-tree set and $f$ is strictly order-reversing, then $[S_n]_f$ is positive definite.
\end{enumerate}
\end{theorem}
\begin{proof}
We prove only the first part since the proof of the second part is dual to it. Let $D_n=\text{meetcl}(S_n)$. We apply Corollary \ref{th:posdefmeet} and show that $\Psi_{D_n,f}(d_k)>0$ for all $d_k\in D_n$. If $k=1$, then $d_k=\min D_n$ and we have $\Psi_{D_n,f}(d_k)=f(d_k)>0$ by assumption. Now let $k>1$. By Lemma \ref{lemma1} $d_k$ covers exactly one element $d_w$ in $\text{meetcl}(S_n)$ and by the order-preserving property we have $f(d_w)<f(d_k)$. Formula \eqref{eq:Psi2} yields
\[
f(d_w)=\sum_{d_v\preceq d_w}\Psi_{D_n,f}(d_v)\ \text{and}\ f(d_k)=\sum_{d_v\preceq d_k}\Psi_{D_n,f}(d_v),
\]
and by subtracting we obtain
\[
0<f(d_k)-f(d_w)=\sum_{d_v\preceq d_k}\Psi_{D_n,f}(d_v)-\sum_{d_v\preceq d_w}\Psi_{D_n,f}(d_v)=\Psi_{D_n,f}(d_k),
\]
which completes the proof.
\end{proof}

As seen in Remark \ref{re:chain}, sometimes it is not only sufficient but also necessary for the positive definiteness of the matrix $(S_n)_f$ that the function $f$ is strictly order-preserving. The next theorem is an example of this. A similar statement can be made regarding join matrices.

\begin{theorem}
If $x_1=\min S_n$ and the Hasse diagram of the set $S_n$ is a tree and the matrix $(S_n)_f$ is positive definite, then the function $f$ is strictly order-preserving in $S_n$ and $f(x_i)>0$ for all $x_i\in S_n$.
\end{theorem}
\begin{proof}
In this case the set $S_n$ is clearly both meet closed and $\wedge$-tree set. We begin the proof by showing that if $x_j$ covers $x_i$, then $f(x_i)<f(x_j)$. By Theorem \ref{th:posdefmeetcl} $\Psi_{S_n,f}(x_j)>0$, and from Equation \eqref{eq:Psi2} we obtain
\[
f(x_j)=\sum_{x_k\preceq x_j}\Psi_{S_n,f}(x_k)\quad\text{and}\quad f(x_i)=\sum_{x_k\preceq x_i}\Psi_{S_n,f}(x_k).
\]
Subtracting the second from the first yields
\[
f(x_j)-f(x_i)=\sum_{x_k\preceq x_j}\Psi_{S_n,f}(x_k)-\sum_{x_k\preceq x_i}\Psi_{S_n,f}(x_k)=\Psi_{S_n,f}(x_j)>0,
\]
from which we obtain $f(x_i)<f(x_j)$. Then suppose that $x_i\prec x_j$ but $x_j$ does not cover $x_i$ for some $x_i,x_j\in S_n$. Since $(P,\preceq)$ and in particular $(S_n,\preceq)$ is locally finite, there is only a finite number of elements of $S_n$ in the interval $[x_i,x_j]$. In fact, by item 3 in Lemma \ref{lemma1}, the elements of the set $S_n\cap[x_i,x_j]$ are always comparable, and therefore the elements of this set form a chain
\[
x_i\prec x_{k_1}\prec x_{k_2}\prec\cdots\prec x_{k_r}\prec x_j
\]
in which the previous element is always covered by the next. This implies that
\[
f(x_i)<f(x_{k_1})<f(x_{k_2})<\cdots<f(x_{k_r})<f(x_j),
\]
and thus we have proven the order-preservation of $f$ in $S_n$ in general.
The second claim now follows easily. By Theorem \ref{th:posdefmeetcl} $f(x_1)=\Psi_{S_n,f}(x_1)>0$. Further, since $x_1\preceq x_i$ for all $x_i\in S_n$ and $f$ is strictly order-preserving, $f(x_i)>0$ for all $x_i\in S_n$.
\end{proof}

In \cite{K2} Korkee studies the meet and join matrices of an $A$-set, which he defines as follows.

\begin{definition}
The set $S_n$ is an $A$-set if the set $A=\{x_i\wedge x_j\ |\ i\neq j\}$ is a chain.
\end{definition}

Korkee derives formulas for the structure, determinant and inverse of the matrix $(S_n)_f$ in a case when $S_n$ in an $A$-set. He also does the same for the matrix $[S_n]_f$ in a case when the dual of $S_n$ is an $A$-set. He does not, however, consider the positive definiteness of these matrices.

It turns out that Theorem \ref{th:posdeftree} can be applied directly to show the positive definiteness of the matrix $(S_n)_f$ when the set $S_n$ is an $A$-set and the positive definiteness of the matrix $[S_n]_f$ when the dual of the set $S_n$ is an $A$-set. This follows from the next theorem.

\begin{theorem}\label{lemma2}
Every $A$-set $S_n$ is also a $\wedge$-tree set and every dual of an $A$-set is a $\vee$-tree set.
\end{theorem}
\begin{proof}
Again we prove only the first part of the claim, since the second part follows from it trivially. Assume that $S_n$ is an $A$-set. First we need to show that $\text{meetcl}(S_n)=S_n\cup A$, where $A$ is the set defined above. In order to do this, we only need to check that $S_n\cup A$ is meet closed. Let $x,y\in S_n\cup A$. We may assume that $x\in S_n$ and $y\in A$, since the other cases are trivial. Now $y=u\wedge v$ for some $u,v\in S_n$, and we obtain
\[
x\wedge y=(x\wedge x)\wedge(u\wedge v)=\underbrace{(x\wedge u)}_{\in A}\wedge \underbrace{(x\wedge v)}_{\in A}\in A,
\]
since $A$ is a chain. Thus the first part of the proof is complete.

Next we prove that every element $x\in\text{meetcl}(S_n)$ covers at most one element of $\text{meetcl}(S_n)$. We now suppose for a contradiction that $x$ covers both $y$ and $z$ for some $x,y,z\in\text{meetcl}(S_n)$. Since $y$ and $z$ are incomparable, we must have $y\not\in A$ or $z\not\in A$ (since $A$ is a chain). We may assume that $z\not\in A$, from which it follows that $z\in S_n$. Now we must have $x\not\in S_n$, since otherwise we would have $x\wedge z=z\not\in A$. Thus $x\in A$ and there exist elements $u,v\in S_n$ such that $x=u\wedge v$. Now, as we can see from Figure \ref{fig: kuva4}, we have $v\wedge z=z\not\in A$, which is a contradiction. The claim now follows from Lemma \ref{lemma1}.
\end{proof}

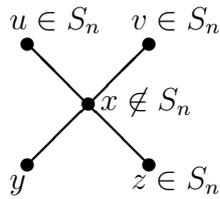
\begin{figure}[ht!]
\centering
\setlength{\unitlength}{0.8 cm}
\begin{picture}(4,4)
\thicklines
\put(1,1){\line(1,1){2}}
\put(3,1){\line(-1,1){2}}

\put(0.7,0.6){$y$}
\put(0.7,3.25){$u\in S_n$}
\put(2.7,0.6){$z\in S_n$}
\put(2.7,3.25){$v\in S_n$}
\put(2.2,1.9){$x\not\in S_n$}

\put(1,1){\circle*{0.2}}
\put(1,3){\circle*{0.2}}
\put(3,1){\circle*{0.2}}
\put(2,2){\circle*{0.2}}
\put(3,3){\circle*{0.2}}
\end{picture}
\caption{Illustration of the proof of Lemma \ref{lemma2}.}
\label{fig: kuva4}
\end{figure}

It is easy to see that the converse of Theorem \ref{lemma2} is not true. Figure \ref{fig: kuva1} exemplifies this. It also illustrates the structure of a typical $A$-set. The Hasse diagram of an $A$-set is always a tree whether the set $S_n$ is finite or not.

\begin{figure}[ht]
\centering
\setlength{\unitlength}{0.8cm}
\begin{picture}(15,9)
\thicklines
\put(5,1){\line(-1,1){4}}
\put(5,1){\line(1,1){3}}
\put(7,3){\line(-1,1){1}}
\put(3,3){\line(1,1){1}}
\put(2,4){\line(0,1){1}}
\put(2,4){\line(1,1){1}}

\put(2.8,5.25){$x_3$}
\put(0.8,5.25){$x_1$}
\put(1.8,5.25){$x_2$}
\put(7.8,4.25){$x_6$}
\put(5.8,4.25){$x_5$}
\put(3.8,4.25){$x_4$}

\put(1,5){\circle*{0.2}}
\put(2,5){\circle*{0.2}}
\put(3,5){\circle*{0.2}}
\put(2,4){\circle*{0.2}}
\put(3,3){\circle*{0.2}}
\put(4,4){\circle*{0.2}}
\put(5,1){\circle*{0.2}}
\put(6,4){\circle*{0.2}}
\put(7,3){\circle*{0.2}}
\put(8,4){\circle*{0.2}}

\put(12,1){\line(0,1){6}}
\put(12,1){\line(1,1){1}}
\put(12,1){\line(2,1){2}}
\put(12,1){\line(-1,1){1}}
\put(12,1){\line(-2,1){2}}
\put(12,3){\line(1,1){1}}
\put(12,5){\line(1,1){1}}
\put(12,5){\line(-1,1){1}}
\put(12,5){\line(-2,1){2}}
\put(12,7){\line(1,1){1}}
\put(12,7){\line(-1,1){1}}

\put(9.75,2.25){$x_1$}
\put(10.75,2.25){$x_2$}
\put(12.75,2.25){$x_3$}
\put(13.75,2.25){$x_4$}
\put(12.75,4.25){$x_5$}
\put(12.1,4.75){$x_6$}
\put(12.75,6.25){$x_9$}
\put(9.75,6.25){$x_7$}
\put(10.75,6.25){$x_8$}
\put(12.75,8.25){$x_{10}$}
\put(10.75,8.25){$x_{11}$}

\put(12,1){\circle*{0.2}}
\put(10,2){\circle*{0.2}}
\put(11,2){\circle*{0.2}}
\put(13,2){\circle*{0.2}}
\put(14,2){\circle*{0.2}}
\put(13,4){\circle*{0.2}}
\put(10,6){\circle*{0.2}}
\put(11,6){\circle*{0.2}}
\put(13,6){\circle*{0.2}}
\put(13,8){\circle*{0.2}}
\put(11,8){\circle*{0.2}}
\put(12,3){\circle*{0.2}}
\put(12,5){\circle*{0.2}}
\put(12,7){\circle*{0.2}}
\end{picture}
\caption{The Hasse diagram on the left is an example of a set $S_6$ that is a $\wedge$-tree set but not an $A$-set. The semilattice on the right is an example of a nontrivial finite $A$-set $S_{11}$.}
\label{fig: kuva1}
\end{figure}
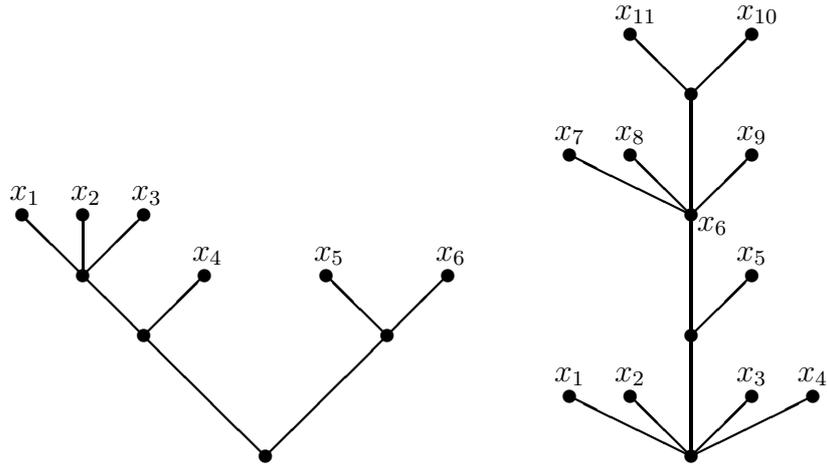

\section{Eigenvalue estimations}

In this section we present bounds for the eigenvalues of certain meet and join matrices. In order to do this, we first need to present the following two lemmas. We here assume that $f$ is strictly order-preserving or order-reversing and also that $f$ is either \emph{increasing} or \emph{decreasing in the set} $S_n$ with respect to the indices $i$ of the elements $x_i$, i.e. $i\leq j\Rightarrow f(x_i)\leq f(x_j)$ or $i\leq j\Rightarrow f(x_i)\geq f(x_j)$. It should be noted that if $f:P\to\Rset$ is either order-preserving or order-reversing, then it is always possible to rearrange the elements of the set $S_n$ so that $f$ becomes increasing or decreasing with respect to the indices. And as stated in Remark \ref{re:order}, this does not affect on the eigenvalues. For example, if $f$ is order-preserving, we may list the images of the elements of the set $S_n$ in ascending order as
\[
f(x_{j_1})\leq f(x_{j_2})\leq\cdots\leq f(x_{j_n}),
\]
and then define $x_i'=x_{j_i}$ for all $i=1,2,\ldots,n$. This even satisfies \eqref{eq:condition}, since by order-preserving property we have
\[
x_i'\preceq x_j'\Rightarrow f(x_i')\leq f(x_j')\Rightarrow i\leq j.
\]
Therefore if the function $f$ is order-preserving, assuming $i\leq j\Rightarrow f(x_i)\leq f(x_j)$ causes no additional restrictions in the study of eigenvalues of meet matrices.

Lemma \ref{upper bound2} and Theorem \ref{Hongtheorem} are generalizations of Hong's and Enoch Lee's results (see \cite[Theorem 2.3]{HE}).

\begin{lemma}\label{upper bound}
Let $f:P\to\Rset$ be a function with nonnegative values, and let $W_k$ denote the $k$-dimensional subspace of the complex vector space $\Cset^n$ consisting of vectors that have zero entries in the coordinates at $k+1,k+2,\ldots,n$ (i.e. $W_k=\mathrm{span}\{\mx{e}_{1},\mx{e}_{2},\ldots,\mx{e}_{k}\}$). Let $\mx{y}=[y_1,\ldots,y_n]^T$ be any vector in $W_k$ (that is, $y_{k+1}=\cdots=y_{n}=0$). If $f$ is order-preserving in $\mathrm{meetcl}(S_n)$, then we have
\begin{equation}\label{eq:lemma1}
\mx{y}^*(S_n)_f\mx{y}\leq k\mx{y}^*\mx{y}f(x_{k}), 
\end{equation}
where $\mx{y}^*$ is the complex conjugate transpose of $\mx{y}$. 
\end{lemma}

\begin{proof}
We apply induction on $k$. In the case when $k=1$ it is rather trivial that
\[
\mx{y}^*(S_n)_f\mx{y}=\overline{y_1}y_1f(x_1)=\mx{y}^*\mx{y}f(x_1),
\]
where $\overline{y_1}$ denotes the complex conjugate of $y_1$. Our induction hypothesis is that the claim holds for $k$ with $1\leq k<n$, and next we  show that the claim also holds for $k+1$.
Let $C_i$ denote the $i$th column of the matrix $(S_n)_f$, and let $\mx{y}\in W_{k+1}$. First we observe that
\[
\mx{y}^*(S_n)_f\mx{y}=\mx{y}^*C_{1}y_{1}+\cdots+\mx{y}^*C_{k}y_{k}+\mx{y}^*C_{k+1}y_{k+1}.
\]
Now let $\mx{z}\in W_k$ such that $z_i=y_i$ for all $i\neq k+1$ and $z_{k+1}=0$. Thus the quadratic form $\mx{z}^*(S_n)_f\mx{z}$ is contained in the previous expression and it can be written as
\begin{align}
\mx{y}^*(S_n)_f\mx{y}=&\mx{y}^*C_{k+1}y_{k+1}+\overline{y_{k+1}}f(x_{k+1}\wedge x_{1})y_{1}+\cdots\notag\\&+\overline{y_{k+1}}f(x_{k+1}\wedge x_{k})y_k+\mx{z}^*(S_n)_f\mx{z}.
\end{align}
Next we start to analyse these terms individually. First of all, the order preserving property of $f$ yields that $0\leq f(x_{k+1}\wedge x_j)\leq f(x_{k+1})$ for all $j=1,\ldots,k$. By also applying the triangle inequality and the simple fact that $|ab|\leq\frac{1}{2}(|a|^2+|b|^2)$ for all $a,b\in\Cset$ we obtain
\begin{align}\label{eq:termi1}
&\left|\mx{y}^*C_{k+1}y_{k+1}\right|\notag\\
&=\left|y_{k+1}\right|\left|\overline{y_{1}}f(x_{k+1}\wedge x_1)+\cdots+\overline{y_{k}}f(x_{k+1}\wedge x_{k})+\overline{y_{k+1}}f(x_{k+1})\right|\notag\\
&\leq \left|y_{k+1}\right|\left(\left|\overline{y_{1}}\right|f(x_{k+1}\wedge x_1)+\cdots+\left|\overline{y_{k}}\right|f(x_{k+1}\wedge x_{k})+\left|\overline{y_{k+1}}\right|f(x_{k+1})\right)\notag\\
&\leq \left|y_{k+1}\right|\left(\left|\overline{y_{1}}\right|+\cdots+\left|\overline{y_{k}}\right|+\left|\overline{y_{k+1}}\right|\right)f(x_{k+1})\notag\\
&\leq\left(\left|y_{k+1}\right|^2+\frac{1}{2}\sum_{i=1}^k\left(\left|y_{k+1}\right|^2+\left|y_{i}\right|^2\right)\right)f(x_{k+1})\notag\\
&=\frac{f(x_{k+1})}{2}\left((k+1)\left|y_{k+1}\right|^2+\mx{y}^*\mx{y}\right).
\end{align}
Very similarly
\begin{align}\label{eq:termi2}
&\left|\overline{y_{k+1}}f(x_{k+1}\wedge x_{1})y_{1}+\cdots+\overline{y_{k+1}}f(x_{k+1}\wedge x_k))y_k\right|\notag\\
&\leq f(x_{k+1})\left|y_{k+1}\right|\left(\left|y_{1}\right|+\cdots+\left|y_{k}\right|\right)\notag\\
&\leq \frac{f(x_{k+1})}{2}\sum_{i=1}^k\left(\left|y_{k+1}\right|^2+\left|y_{i}\right|^2\right)
=\frac{f(x_{k+1})}{2}\left((k-1)\left|y_{k+1}\right|^2+\mx{y}^*\mx{y}\right).
\end{align}
Finally, our induction hypothesis and the increase of $f$ in the set $S_n$ with respect to the indices $i$ yields
\begin{equation}\label{eq:termi3}
\mx{z}^*(S_n)_f\mx{z}\leq k\mx{z}^*\mx{z}f(x_{k})\leq k\mx{z}^*\mx{z}f(x_{k+1}).
\end{equation}
Now, by combining \eqref{eq:termi1}, \eqref{eq:termi2} and \eqref{eq:termi3} we obtain
\begin{align*}
\left|\mx{y}^*(S_n)_f\mx{y}\right|&\leq \frac{f(x_{k+1})}{2}\left((k+1)\left|y_{k+1}\right|^2+\mx{y}^*\mx{y}\right)\notag\\&\hspace{2cm}+\frac{f(x_{k+1})}{2}\left((k-1)\left|y_{k+1}\right|^2+\mx{y}^*\mx{y}\right)+k\mx{z}^*\mx{z}f(x_{k+1})\\
&=f(x_{k+1})\left(\mx{y}^*\mx{y}+\underbrace{k|y_{k+1}|^2+k\mx{z}^*\mx{z}}_{=k\mx{y}^*\mx{y}}\right)=(k+1)f(x_{k+1})\mx{y}^*\mx{y}.
\end{align*}
This completes the proof.
\end{proof}

\begin{lemma}\label{upper bound2}
Let $f:P\to\Rset$ be a function with nonnegative values, and let $V_k$ denote the $k$-dimensional subspace of the complex vector space $\Cset^n$ consisting of vectors that have zero entries in the coordinates at $1,2,\ldots,n-k$ (i.e. $V_k=\mathrm{span}\{\mx{e}_{n-k+1},\mx{e}_{n-k+2},\ldots,\mx{e}_{n}\}$). Let $\mx{y}=[y_1,\ldots,y_n]^T$ be any vector in $V_k$ (that is, $y_1=\cdots=y_{n-k}=0$). If $f$ is order-reversing in $\mathrm{joincl}(S_n)$, then
\begin{equation}\label{eq:lemma2}
\mx{y}^*[S_n]_f\mx{y}\leq k\mx{y}^*\mx{y}f(x_{n-k+1}).
\end{equation}
\end{lemma}

\begin{proof}
The proof is very similar to the proof of Lemma \ref{upper bound} and is essentially the same as Hong and Enoch Lee's proof in \cite[Theorem 2.3]{HE}.
\end{proof}

By applying the Courant-Fischer theorem together with Lemmas \ref{upper bound} and \ref{upper bound2} we are now able to give bounds for the eigenvalues of the matrices $(S_n)_f$ and $[S_n]_f$.

\begin{theorem}\label{ominaisarvoraja}
Let $\lambda^{(n)}_1,\lambda^{(n)}_2,\ldots,\lambda^{(n)}_n$, where $\lambda^{(n)}_1\leq\lambda^{(n)}_2\leq\cdots\leq\lambda^{(n)}_n$, denote the eigenvalues of the matrix $(S_n)_f$. Under the assumptions of Lemma \ref{upper bound} we have
\begin{equation}
\lambda^{(n)}_k\leq kf(x_k)
\end{equation}
for all $k=1,\ldots,n$. Moreover, $f(x_n)\leq\lambda^{(n)}_n$.
\end{theorem}

\begin{proof}
Let $1\leq k\leq n.$ By applying Lemma \ref{upper bound} and Courant-Fischer theorem (\cite[Theorem 4.2.11]{HJ}) we obtain
\begin{align*}
kf(x_k)&\geq \max_{\mx{0}\neq\mx{y}\in W_k}\frac{\mx{y}^*(S_n)_f\mx{y}}{\mx{y}^*\mx{y}}=\max_{\mx{0}\neq\mx{y}\perp \mx{e}_{k+1},\ldots,\mx{e}_n}\frac{\mx{y}^*(S_n)_f\mx{y}}{\mx{y}^*\mx{y}}\\
&\geq \min_{\mx{w}_{1},\mx{w}_{2},\ldots,\mx{w}_{n-k}\in\Cset^n}\left(\max_{\mx{0}\neq\mx{y}\perp \mx{w}_{1},\mx{w}_2,\ldots,\mx{w}_{n-k}}\frac{\mx{y}^*(S_n)_f\mx{y}}{\mx{y}^*\mx{y}}\right)=\lambda^{(n)}_k.
\end{align*}
The rest of the claim follows from the Rayleigh-Ritz theorem (\cite[Theorem 4.2.2]{HJ}) by setting $\mx{y}=\mx{e}_n$, since
\[
\lambda^{(n)}_{n}=\max_{\mx{y}\neq\mx{0}}\frac{\mx{y}^*(S_n)_f\mx{y}}{\mx{y}^*\mx{y}}\geq\frac{\mx{e}_n^*(S_n)_f\mx{e}_n}{\mx{e}_n^*\mx{e}_n}=f(x_n).
\]
\end{proof}

\begin{theorem}\label{Hongtheorem}
Let $\hat{\lambda}^{(n)}_1,\hat{\lambda}^{(n)}_2,\ldots,\hat{\lambda}^{(n)}_n$, where $\hat{\lambda}^{(n)}_1\leq\hat{\lambda}^{(n)}_2\leq\cdots\leq\hat{\lambda}^{(n)}_n$, denote the eigenvalues of the matrix $[S_n]_f$. Under the assumptions of Lemma \ref{upper bound2} we have
\begin{equation}
\hat{\lambda}^{(n)}_k\leq kf(x_{n-k+1})
\end{equation}
for all $k=1,\ldots,n$. In addition, $f(x_1)\leq\hat{\lambda}^{(n)}_n$.
\end{theorem}

\begin{proof}
The proof is similar to the proof of Theorem \ref{ominaisarvoraja}.
\end{proof}

\begin{example}
Let $\alpha\in\Rset^+$, $S_n=\{x_1,\ldots,x_n\}\subset\Zset^+$ and $f:\Zset^+\to\Rset$ be the function such that $f(n)=n^{\alpha}$ for all $n\in\Zset^+$. Let $(S_n)_f$ denote the power GCD matrix having $(x_i,x_j)^\alpha$ as its $ij$ element, and let $(S_n)_f^{**}$ denote the power GCUD matrix having $((x_i,x_j)^{**})^\alpha$, the power of the greatest common unitary divisor of $x_i$ and $x_j$ as its $ij$ element ($d$ divides $x_i$ unitarily if $d\,|\,x_i$ and $(d,x_i/d)=1$). Both these matrices fulfill the assumptions of Lemma \ref{upper bound}, and therefore by Theorem \ref{ominaisarvoraja} $kf(x_k)=kx_k^\alpha$ is an upper bound for the $k$th largest eigenvalue of both $(S_n)_f$ and $(S_n)_f^{**}$. Moreover, $f(x_n)=x_n^{\alpha}$ is a lower bound for the largest eigenvalue of both $(S_n)_f$ and $(S_n)_f^{**}$.
\end{example}

\begin{example}[\cite{HE},Theorem 2.3]
Let $\alpha\in\Rset^+$, $S_n=\{x_1,\ldots,x_n\}\subset\Zset^+$ and $f:\Zset^+\to\Rset$ be the function such that $f(n)=\frac{1}{n^{\alpha}}$ for all $n\in\Zset^+$. In this case the matrix $[S_n]_f$ having $\frac{1}{[x_i,x_j]^\alpha}$ as its $ij$ element is referred to as the reciprocal power LCM matrix of the set $S_n$. Let $\lambda_k^{(n)}$ denote the $k$th largest eigenvalue of the matrix $[S_n]_f$. Thus by Theorem \ref{Hongtheorem} we have
\[
\lambda_k^{(n)}\leq kf(x_{n-k+1})=\frac{k}{x_{n-k+1}^\alpha}.
\]
In addition, $f(x_1)=\frac{1}{x_1^\alpha}\leq\lambda_n^{(n)}$.
\end{example}

\end{document}